\documentclass[11pt,draft,a4paper,DIV10]{scrartcl}

\usepackage{amsmath}
\usepackage{amssymb}
\usepackage{amsthm}
\usepackage{xcolor}

\usepackage[english]{babel}
\usepackage{ifthen}
\usepackage{booktabs}
\usepackage{wasysym}
\usepackage{multirow}
\usepackage{xspace}
\usepackage{enumerate}
\usepackage{url}

\usepackage[final]{hyperref}

\usepackage{tikz}
\usetikzlibrary{shapes}

\definecolor{darkgreen}{rgb}{0.0,0.7,0.0}
\newenvironment{bs}{\noindent\color{darkgreen} BS:}{}

\newenvironment{mk}{\noindent\color{blue} MK:} {}

\newenvironment{vd}{\noindent\color{red} VD:} {}

\newcommand{\refthm}[1]{Theorem~\ref{#1}\xspace}
\newcommand{\refcor}[1]{Corollary~\ref{#1}\xspace}

\newcommand{\reflem}[1]{Lemma~\ref{#1}\xspace}
\newcommand{\refprp}[1]{Proposition~\ref{#1}\xspace}
\newcommand{\refrem}[1]{Remark~\ref{#1}\xspace}

\newcommand{\refsec}[1]{Section~\ref{#1}\xspace}

\newcommand{\refex}[1]{Example~\ref{#1}\xspace}

\newcommand{\IFF}{if and only if\xspace}
\newcommand{\homo}{homomorphism\xspace}

\newcommand{\tm}{transformation monoid\xspace}

\newcommand{\sse}{\subseteq}

\newcommand{\ov}{\overline}

\renewcommand{\phi}{\varphi}

\newcommand{\sig}{\sigma}

\newcommand{\set}[2]{\left\{#1\mathrel{\left|\vphantom{#1}\vphantom{#2}\right.}#2\right\}}
\newcommand{\oneset}[1]{\left\{\mathinner{#1}\right\}}
\newcommand{\smallset}[1]{\left\{#1\right\}}

\newcommand{\abs}[1]{\left|\mathinner{#1}\right|}

\newcommand{\gen}[1]{\left\langle\mathinner{#1}\right\rangle}

\newcommand{\N}{\mathbb{N}}

\newtheorem{theorem}{Theorem}[section]

\newtheorem{proposition}[theorem]{Proposition}
\newtheorem{lemma}[theorem]{Lemma}
\newtheorem{corollary}[theorem]{Corollary}
\newtheorem{remrk}[theorem]{Remark}
\newenvironment{remark}{\begin{remrk}\upshape}{\end{remrk}}
\newtheorem{expl}[theorem]{Example}
\newenvironment{example}[1][]{\ifthenelse{\equal{#1}{}}{\begin{expl}\upshape}{\begin{expl}[#1]\upshape}}{\end{expl}}

\newcommand{\EEX}{\hspace*{\fill}\ensuremath{\Diamond}} % End EXample
\newcommand{\ERM}{\hspace*{\fill}\ensuremath{\Diamond}} % End ReMark

\newcommand{\dotcup}{\mathbin{\dot{\cup}}}

\begin{document}

%\setcounter{page}{1001}
%\issue{XXI~(2011)}

\title{The Krohn-Rhodes Theorem \\ and Local Divisors}

\author{Volker Diekert\thanks{Institut f\"ur Formale Methoden der Informatik, Universit{\"a}t Stuttgart, Germany} \and
  Manfred Kuf\-leitner$^{*,}$\thanks{The second author was
    supported by the German Research Foundation (DFG) under grant
    \mbox{DI 435/5-1}.} \and
  Benjamin Steinberg\thanks{Department of Mathematics, City College of New York, USA}}

% \author{Volker Diekert \\
% {\small Institut f\"ur Formale Methoden der Informatik} \\[-1.9mm]
% {\small Universit{\"a}t Stuttgart} \\[-1.9mm]
% {\small Stuttgart, Germany} \\[-1.9mm]
% {\small \texttt{diekert{@}fmi.uni-stuttgart.de}} \\
% \and Manfred Kuf\-leitner\thanks{The second author was
%   supported by the German Research Foundation (DFG) under grant
%   \mbox{DI 435/5-1}} \\
% {\small Institut f\"ur Formale Methoden der Informatik} \\[-1.9mm]
% {\small Universit{\"a}t Stuttgart} \\[-1.9mm]
% {\small Stuttgart, Germany} \\[-1.9mm]
% {\small \texttt{kufleitner{@}fmi.uni-stuttgart.de}} \\
% \and Benjamin Steinberg \\
% {\small Department of Mathematics} \\[-1.9mm]
% {\small City College of New York} \\[-1.9mm]
% {\small New York, USA} \\[-1.9mm]
% {\small \texttt{bsteinberg{@}ccny.cuny.edu}}
% }

\date{}

\maketitle

\begin{abstract}
  \noindent
  \textbf{Abstract:}
  We give a new proof of the Krohn-Rhodes theorem using local
  divisors. The proof provides nearly as good a decomposition in terms
  of size as the holonomy decomposition of Eilenberg, avoids induction
  on the size of the state set, and works exclusively with monoids
  with the base case of the induction being that of a group.

  \medskip

\end{abstract}

\section{Introduction}

The Krohn-Rhodes theorem is one of the fundamental results in finite
semigroup theory.  It asserts that a finite semigroup can be
decomposed in a suitable sense into a wreath product of
well-controlled finite simple groups and copies of a $3$-element
monoid, all of whose elements are idempotents.  In addition to the
elegance of the result, it has a number of important applications.
For instance, it can be used as an induction scheme for proving
Sch\"utzenberger's celebrated theorem on star-free
languages~\cite{sch65sf}, as well as generalizations due to
Straubing~\cite{str79}.

There are a number of proofs of the theorem in the literature, see
e.g.~\cite{Eil76,esik00,ginz68,hol82,krt68arbib8,Lallementproof,Lallementfixed,%
mt69mst,rs09qtheory,str94,Zeiger}.
They tend to fall into three classes.  The first class consists of
fundamentally algebraic proofs.  These proofs follow the line of proof
of Krohn and Rhodes from~\cite{krt68arbib8}.  They are based on
induction on the size of the semigroup and rely on the non-obvious
Trichotomy Lemma of Krohn and Rhodes which states that a finite
semigroup is either left simple, cyclic or can be written as $V\cup T$
where $V$ is a proper left ideal and $T$ is a proper subsemigroup.
The basis of the induction then becomes left simple semigroups and
cyclic groups.  These proofs use transformation semigroups only to
adjoin constant maps and to keep the wreath product decompositions
simpler.  The state set of the transformation semigroup does not play
a role in the induction. Lallement attempted an approach using
exclusively monoids and avoiding transformation
representations~\cite{Lallementproof}, but it had a flaw that can only
be fixed by passing to transformation monoids~\cite{Lallementfixed}.

The second type of proof uses transformation semigroups and performs
an induction based on both semigroup size and state size.  A typical
example is the first of the two proofs of the Krohn-Rhodes theorem
given by Eilenberg in~\cite{Eil76}.  The induction scheme relies on
working with semigroups rather than monoids and essentially throws the
Trichotomy Lemma into the states.  These proofs tend to have a very
large blow up in the state size.  This approach also tends to be more
technical (and less conceptual) than the algebraic approach above.

The third type of proof is based on Zeiger's method~\cite{Zeiger}.
The definitive form of this approach is the holonomy theorem of
Eilenberg~\cite{Eil76}.  This approach does not use induction at all.
It provides instead a direct decomposition of a transformation
semigroup into a wreath product of permutation groups with adjoined
constant maps.  This proof keeps better control over the state set and
seems to give the most efficient decomposition.  It is however the
most technical argument.

Our proof is based on the notion of local divisors and encompasses
features of all the above proofs.  Like Lallement's proof, we avoid
semigroups.  Like the algebraic proofs, we avoid putting the state set
into the induction scheme.  However, we avoid the Trichotomy Lemma:
the base case of our argument is that of groups.  A local divisor is a
monoidal generalization of a Sch\"utzenberger group; the latter
appears implicitly in the holonomy proof.

The concept of a local divisor is an old one. In commutative algebra
it has been introduced by Meyberg in 1972, see~\cite{FeTo02,Mey72}.
In finite semigroup theory and formal language the explicit definition
of a local divisor was first given in~\cite{dg06IC}. It was used for
proving that local temporal logics are expressively complete for
partially commutative monoids~\cite{dg06IC}. In \cite{dg08SIWT:short}
it has been used for the same purpose in the context of finite and infinite words. 
A category
generalization is being used by Costa and the third author in the
context of symbolic dynamics (unpublished).

The key idea of local divisors is the following. If $e$ is an
idempotent of a monoid $M$, then $eMe=eM\cap Me$ is a subsemigroup
which is a monoid with identity $e$.  The group of units of $eMe$ is
the $\mathcal H$-class of $e$.  If $M$ acts faithfully on the right of
a set $X$, then $eMe$ acts faithfully on the right of $Xe$.
Sch\"utzenberger famously put a group structure on the $\mathcal
H$-class of an arbitrary element $c\in M$ with $c$ as the
identity~\cite{cp67,rs09qtheory}. The local divisor construction
extends this, by a multiplication $\circ$, to a monoid structure on $cM\cap Mc$ whose group of units
is the Sch\"utzenberger group.  
%In various applications it is enough 
%to apply  this local divisor $cM\cap Mc$ construction  to aperiodic
%monoids~\cite{dg06IC,dg08SIWT:short}, where $\mathcal H$-classes are trivial. 
 The monoid $(cM\cap Mc, \circ)$ acts faithfully on $Xc$
whenever $M$ acts faithfully on $X$.  We shall use the submonoid
$(cMc\cup \{c\},\circ)$ in this paper since it might have a smaller cardinality
than $cM\cap Mc$.

\section{Preliminaries}\label{sec:prem}

Apart from basic knowledge about monoids, we have tried to keep this
paper self-contained. In particular, we give a short introduction to
transformation monoids in \refsec{sec:vd1}; and we present the wreath
product construction and some of its basic properties in
Sections~\ref{sec:vd2} through~\ref{vd4}.  These first few sections of
the preliminaries are only slight adaptations of standard
notions. \refsec{ssc:vdloc} explains the notion of local divisor and
its basic properties.  It is a main tool used in our proof of the
Krohn-Rhodes Theorem.  Some standard proofs from this section are
relegated to the appendix, \refsec{app}.

%\section{VD}\label{VD}
\subsection{Transformation monoids}\label{sec:vd1}

A \emph{right action} of a monoid $M$ on a
nonempty %Vd OK, but it seems that a
%natural faithfully full embedding of monoids into \tm{}s uses $(\emptyset,M)$.
set $X$ is
a % (total) %mappings are always total if not otherwise stated
mapping $X \times M \to X$, written $(x,m) \mapsto x \cdot m$,
satisfying $x \cdot 1 = x$ and $(x \cdot m_1) \cdot m_2 = x \cdot (m_1
m_2)$ for all $x \in X$ and $m_1, m_2 \in M$.  An action is therefore
the same thing as a \homo of $\sig\colon M \to T(X)$. Here and in the
following $T(X)$ denotes the monoid of all mappings from $X$ to
itself. It is a monoid by letting $fg = h$ where $h$ is defined by
$h(x) = g(f(x))$ for $f,g \in T(X)$ and $x \in X$.  The action of $M$
is \emph{faithful} if $\sig$ is injective. Thus, the action is
faithful, if each $m \in M$ is uniquely determined by knowing $ x
\cdot m$ for all $x \in X$.  The pair $(X,M)$ is called a
\emph{transformation monoid}. We do not require the action of
transformation monoids to be faithful. \emph{Left actions} are defined
symmetrically.

The pair $(\oneset{1},M)$ is a transformation monoid, but the action
is not faithful unless $M$ is trivial. The multiplication defines the
faithful transformation monoid $(M,M)$.  The pair $(X,T(X))$ is
another faithful transformation monoid.  The monoid $T(X)$ acts on the
right, which justifies to write $xf$ instead of $f(x)$. This
suffix-notation turns out to be convenient in the following.

A \emph{morphism} between transformation monoids $(Y,N)$ and $(X,M)$
is a pair $(\phi,\psi)$ where $\varphi\colon Y \to X$ is a mapping and
$\psi \colon N \to M$ is a \homo such that $\varphi(y \cdot n) =
\varphi(y) \cdot \psi(n)$ for all $y \in Y$ and $n \in N$. It is
called \emph{surjective} (resp.{} \emph{injective}) if both mappings
$\phi$ and $\psi$ are surjective (resp.{} injective).  It is an
isomorphism, if $\phi$ and $\psi$ are bijections.  The \tm $(X,M)$ is
called \emph{finite} if both $X$ and $M$ are finite.
%For inclusions (injections) we occasionally write $(X',M') \leq (X,M)$
%and $M' \leq M$.
%\vdd{Actually we do it for $M' \leq M$ only. I propose to  remove the
%notation $M' \leq M$ and the sentence above.}

\subsection{Wreath products}\label{sec:vd2}

Let $(X,M)$ and $(Y,N)$ be transformation monoids. For a moment let
$+$ denote the multiplication in $M$ (which might be
non-commutative). The functions $M^Y$ from $Y$ to $M$ form a
monoid by componentwise multiplication. That is, $f+g$ is defined by
$y(f+g) = yf + yg$ for $y\in Y$ and $f,g \in M^Y$.  As $N$ acts on $Y$
on the right, it induces a left-action $*$ of $N$ on $M^Y$
defined by:
\begin{equation*}
  y (n* f) = (y\cdot n)f.
\end{equation*}
With this definition we turn the set $M^Y\times N$ into a monoid as a
\emph{semidirect product} denoted by $M^Y \rtimes N$. The
multiplication in $M^Y \rtimes N$ is given by:
\begin{equation*}
  (f,n)\cdot (g,k) = (f+(n *g),nk).
\end{equation*}
The multiplication is associative and we have
$(f,n)\cdot (g,k)\cdot (h,\ell) = (f+(n *g) + (nk*h),nk\ell).$ The
\emph{wreath product} $M \wr (Y,N)$ is this semidirect product.
The monoid $M \wr (Y,N)$ acts on $X \times Y$ by $(x,y) \cdot (f,n) =
(x \cdot yf,y \cdot n)$. The resulting transformation monoid is
denoted by $(X,M) \wr (Y,N)$.  It is the \emph{wreath product} of the
transformation monoids $(X,M)$ and $(Y,N)$. If both $(X,M)$ and
$(Y,N)$ are faithful, then $(X,M) \wr (Y,N)$ is faithful, too.  The
wreath product of transformation monoids is associative up to
isomorphism. This fact is well-known and stated as \reflem{lem:assoc}.
The proof follows from a straightforward calculation based on the
canonical bijection %\begin{equation*}
$(M^Y \times N)^Z \to M^{Y\times Z} \times N^{Z}$.
        % \end{equation*}
Details can be found  in the appendix, \refsec{app}.

\begin{lemma}\label{lem:assoc}
%  The transformation monoids
  $\big( (X,M) \wr (Y,N) \big) \wr (Z,P)$ and $(X,M) \wr \big( (Y,N)
  \wr (Z,P) \big)$ are isomorphic.
\end{lemma}

\begin{remark}\label{rem:mon}
  The number (of isomorphism classes) of finite transformation monoids
  is countable.  The wreath product operation turns this countable set
  into an infinite monoid with $(\smallset{1}, \smallset{1})$ as a
  neutral element.  According to a standard convention, we define the
  wreath product over an empty index set as the trivial \tm
  $(\smallset{1}, \smallset{1})$.
  \ERM
\end{remark}

\subsection{Divisors}\label{sec:vd3}

A monoid $M$ \emph{divides} a monoid $N$, written as $M \prec N$, if
$M$ is the homomorphic image of a submonoid of $N$.  This notion
immediately extends to \tm{}s: A \tm $(X,M)$ \emph{divides} $(Y,N)$,
if there exists a \tm $(Y',N')$ together with a surjective morphism
from $(Y',N')$ onto $(X,M)$ and an injective morphism from $(Y',N')$
into $(Y,N)$. In particular, $(X,M) \prec (Y,N)$ implies $M \prec
N$. The divisor realation yields a partially defined surjection from
$Y$ to~$X$ where the domain is $Y'\sse Y$.  However, the Krohn-Rhodes
decomposition holds for totally defined surjections from $Y$ to
$X$. Thus, it is enough (and more convenient here) to restrict
ourselves to totally defined surjections.  For this we introduce the
notion of strong division: We say that $(X,M)$ \emph{strongly divides}
$(Y,N)$, if there exists a submonoid $N'$ of $N$ and a surjective
morphism from $(Y,N')$ onto $(X,M)$. In this case we also say that
$(X,M)$ is a \emph{strong divisor} of $(Y,N)$ and we write $(X,M)
\prec (Y,N)$. Every strong divisor is a divisor. This is why our
notation $(X,M) \prec (Y,N)$ is ``on the safe side''.  Another way to
express $(X,M) \prec (Y,N)$ is that there exists a surjection $\varphi
: Y \to X$, a submonoid $N'$ of $ N$, and a surjective homomorphism
$\psi \colon N' \to M$ such that $\varphi(y) \cdot \psi(n) = \varphi(y
\cdot n)$.

The notion of divisor is closely related to Eilenberg's notion of
a covering \cite{Eil76}. As we use here strong division, we restrict
the definition of a covering to totally defined surjections: Let
$(X,M)$ and $(Y,N)$ be transformation monoids and $\varphi \colon Y
\to X$ be any surjective mapping.  An element $\widehat{m} \in N$ is
called a \emph{cover} of $m\in M$ if $\varphi(y) \cdot m = \varphi(y
\cdot \widehat{m})$ for all $y\in Y$.  This defines a submonoid
$S_\phi$ of $M\times N$ %(or relational morphism {}from $M$ to $N$)
as follows:
\begin{equation*}
  S_\phi % = \set{(m,n) \in M\times N}{\forall y \in Y: \,  \varphi(y) \cdot
%m = \varphi(y \cdot n)}
= \set{(m,\widehat{m}) \in M\times N}{\widehat{m} \text{ covers  $m$}}.
\end{equation*}
The intuition is that $(m,\widehat{m}) \in S_\varphi$ says that
$\widehat{m}$ can ``simulate'' $m$ in the sense that, instead of
computing $\varphi(y) \cdot m$ in $(X,M)$ we can do the computation $y
\cdot \widehat{m}$ in $(Y,N)$ and then apply $\varphi$.  Let $\pi_i
\colon M \times N \to M$ is the projection to the $i$-th component, $i
= 1,2$.  We say that $\varphi$ is a \emph{covering} if
$\pi_1(S_\varphi) = M$, i.e., every $m\in M$ has some cover.

It turns out that $\varphi$ defines a division $(X,M) \prec (Y,N)$
\IFF there is submonoid $R_\phi \sse S_\phi $ such that $\pi_1\colon
R_\phi \to M$ is surjective and $\pi_2\colon R_\phi \to N$ is
injective.  Indeed, if $(X,M) \prec (Y,N)$ is due to a pair
$(\phi,\psi)$, then we can choose $R_\phi = \set{(\psi(n), n)}{n \in
  N'}$. The other way round, let $R_\phi \sse S_\phi $ such that
$\pi_1\colon R_\phi \to M$ is surjective and $\pi_2\colon R_\phi \to
N$ is injective. Then we obtain a surjective \homo $\psi$ from $N'=
\pi_2(R_\phi)$ onto $M$. It follows that the pair $(\phi,\psi)$ is a
surjective morphism of $(Y,N')$ onto $(X,M)$, and the \tm $(Y,N')$
embeds into $(Y,N)$.

The following proposition collects some useful properties and
relations between coverings and strong divisors. In particular, for
faithful transformation monoids the notions of covering and strong
divisor become equivalent.

\begin{proposition}\label{cover}
  Let $(X,M)$ and $(Y,N)$ be transformation monoids and let $\varphi
  \colon Y \to X$ be surjective.
  \begin{enumerate}
  \item\label{aaa:cover} If $M$ is generated by $A \sse M$ and every
    $a\in A$ has a cover $\widehat{a} \in N$, then $\varphi \colon Y
    \to X$ is a covering.
  \item\label{bbb:cover} If $\varphi \colon Y \to X$ is a covering and
    $(X,M)$ is faithful, then $(X,M) \prec (Y,N)$. In
    particular, $M \prec N$.
  \end{enumerate}
\end{proposition}

\begin{proof}
  Assertion ``\ref{aaa:cover}.'' is trivial, because
  $\widehat{a_1}\cdots \widehat{a_m}$ is a cover of ${a_1}\cdots
  {a_m}$.  To see assertion ``\ref{bbb:cover}.'' it suffices to show
  that $\pi_2\colon S_\varphi \to N$ is injective.  Suppose
  $(m_1,n),(m_2,n) \in S_\varphi$. Then $\varphi(y) \cdot m_1 =
  \varphi(y \cdot n) = \varphi(y) \cdot m_2$ for all $y \in Y$. Since
  $\varphi$ is surjective and $(X,M)$ is faithful, we conclude $m_1 =
  m_2$.
%
%  \QED
\end{proof}

\begin{example}\label{exex}
  We have the following divisions.
  \begin{itemize}
%  \item If $Y \subseteq X$ with $Y \cdot M \subseteq Y$, then $(Y,M)
%    \prec (X,M)$.
%    %Here, $\varphi$ and $\psi$ are the identities on $Y$
%    and $M$, respectively. If $X \neq Y$, then this division is not
%    strong.
  \item Every \tm $(X,M)$ strongly divides $(X \times M,M)$ equipped
    with the faithful action $(x,m) \cdot m' = (x, m m')$. In this
    situation, $\varphi (x,m) = x \cdot m$ and $\psi$ is the identity
    on $M$.
  \item $(\smallset{1},M)$ is covered by
    $(\smallset{1},\smallset{1})$, and $(\smallset{1},M)$ strongly
    divides the faithful transformation monoid $(M,M)$.
  \item If $N$ is a submonoid of $M$, then $(X,N)$ strongly divides
    $(X,M)$.%; here, $\varphi$ and $\psi$ are the indentities on $X$ and $N$, respectively.
  \item Every direct product $(X,M) \times (Y,N) = (X \times Y,M
    \times N)$ strongly divides $(X,M) \wr (Y,N)$. Indeed, let
    $\varphi$ be the identity on $X \times Y$ and let $P \subseteq
    M^Y$ contain all constant functions $k_m$ with $yk_m = m$ for all
    $y \in Y$. The subset $P \times N$ is a submonoid in the
    semidirect product $M^Y \rtimes N$. We see that $R_\phi =
    \set{((m,n),(k_m,n))}{m\in M, n\in N}$ is a subset of $S_\phi$. Moreover,
    $R_\phi$ satisfies the conditions that $\pi_1$ is surjective and
    $\pi_2$ is injective. 
    \EEX
  \end{itemize}
\end{example}

We conclude this section with a few more well-known facts.  For proofs
see again \refsec{app}.
%For example, wreath products are compatible with division:

\begin{lemma}\label{lem:wrdiv}
  If $(X,M) \prec (X',M')$ and $(Y,N) \prec (Y',N')$, then
  $(X,M) \wr (Y,N) \prec (X',M') \wr (Y',N')$.
\end{lemma}

\begin{remark}\label{rem:mon2}
  The set (of isomorphism classes) of finite transformation monoids is
  partially ordered by the divisor relation $\prec $. This ordering is
  compatible with the multiplication by the wreath product operation
  by \reflem{lem:wrdiv}. Hence this set is an infinite ordered monoid.
  \ERM
\end{remark}

\begin{proposition}\label{prp:dg}
  If $N$ is a normal subgroup of $G$, then $(G,G) \prec (N,N) \wr
  (G/N,G/N)$.
\end{proposition}

A group $G$ is \emph{simple} if for every normal subgroup $N$ of $G$
we have $N=\smallset{1}$ or $N = G$. If $G$ is finite but not simple,
then (by induction) there exists a proper normal subgroup $N$ of $G$
such that $N$ is smaller and $G/N$ is non-trivial.

\begin{corollary}\label{sg}
  For every finite group $G$, the transformation monoid $(G,G)$
  strongly divides a wreath product of the form
  \begin{equation*}
    (G_1,G_1) \wr \cdots \wr (G_m,G_m)
  \end{equation*}
  where each $G_i$ is a simple group dividing $G$. Moreover,
  $\abs{G_1} \cdots \abs{G_m} = \abs{G}$.
\end{corollary}

\begin{proof}
  This follows by induction using \refprp{prp:dg} and the classical
  formula $\abs{N} \cdot \abs{G/N} = \abs{G}$ for subgroups $N$ in
  $G$.
\end{proof}

\subsection{Constants} \label{vd4}

Let $X$ be a set and let (as above) $T(X)$ be the monoid of all
mappings from $X$ to itself. For $x \in X$ we denote by $\ov x$ the
constant function which maps all $y \in X$ to $x$, i.e., $y\ov x = x$
for all $x,y \in X$.  By $\ov X$ we denote the subset $\set{\ov x}{x
  \in X}$ of $T(X)$.  This defines the faithful transformation monoid
$(X,U_X)$ with $U_X = \overline{X} \cup \smallset{1}$. In semigroup
theory, $U_n$ denotes the monoid resulting from adjoining an external
identity to the $n$-element right zero semigroup
$\ov{\{1,\ldots,n\}}$.  In particular, $U_1$ is the two-element
monoid.  Note that if $|X|=n\geq 2$, then $U_X\cong U_n$.  On the
other hand, $U_1$, which is a submonoid of $U_2$, is not of the form
$U_X$ for any set $X$.

Now let $(X,M)$ be any faithful \tm. Viewing $M$ as a submonoid of
$T(X)$ we can define $M\cup \ov X \sse T(X)$. A straightforward
verification shows that $M\cup \ov X$ is a submonoid of $T(X)$:
Indeed, we have $m\ov x = \ov x$ and $\ov x \, m= \ov {x\cdot m}$.
For $\abs X =1$ we have $M\cup \ov X= M = \smallset{1}$, because the
action is faithful. For $n = \abs X >1$ the monoid $U_n$ is a
submonoid of $M\cup \ov X $. We define
\begin{equation*}
  \ov{(X,M)}= (X,M\cup \ov X).
\end{equation*}
In particular, $\ov{(\smallset{1},\smallset{1})} =
(\smallset{1},\smallset{1})$.  Another way to think about $\ov{(X,M)}$
is that all \emph{missing constants\,}Êhave been adjoined to
$(X,M)$. In this sense we can view $\ov{(X,M)}$ as a closure of
$(X,M)$.  We have $\ov{(X,M)}= \ov{\,\ov{(X,M)}\,}$.

Again, we conclude with a few more well-known results whose proofs can
be found in \refsec{app}.

\begin{lemma}\label{lem:dgwc}
  Let $(X,G)$ be a faithful transformation monoid such that $G$ is a
  group. Then 
  \begin{equation*}
    \overline{(X,G)} \prec (X,U_X) \wr (G,G).
  \end{equation*}
\end{lemma}

\begin{lemma}\label{lem:dUn}
  $(\smallset{a_0,\ldots,a_n},U_{n+1}) \prec (\smallset{a_0, \ldots,
    a_{n-1}},U_{n}) \times (\smallset{a_0,a_n},U_2)$.
\end{lemma}

\subsection{Local divisors}
\label{ssc:vdloc}

Let $M$ be a monoid and $c \in M$.  We have the following inclusions
of subsemigroups:
\begin{equation*}
  cMc \subseteq cMc \cup \smallset{c} \subseteq cM \cap Mc.
\end{equation*}
If $c$ is a unit of $M$, then $M = cMc$. Otherwise, if $c$ is not a
unit, then $1 \notin cM\cap Mc$ and thus $cM\cap Mc \neq M$.  If $c =
c^2$ is idempotent, then $c \in cMc = cM\cap Mc$ and $cMc$ is the
so-called local monoid at the idempotent $c$.  We generalize this
notion to arbitrary elements $c$ by introducing a new multiplication
$\circ$ on $cM \cap Mc$. We let
\begin{equation*}
  mc \circ cn = mcn.
\end{equation*}
This operation is well-defined since $m'c = mc $ and $cn' = cn$
implies $m' c n' = mc n' = m cn$. For $mc, nc \in cM$ we have $mc
\mathop{\circ} nc = mnc \in cM$. Thus, $\circ$ is associative and $c$
is neutral. In particular, $(cM \cap Mc, \circ, c)$ forms a monoid and
$(cMc \cup \smallset{c}, \circ, c)$ is a
submonoid. %; moreover, $(cMc,\circ)$ is a semigroup.
The set $M' = \set{m\in M}{mc \in cM}$ is a submonoid of $M$ and $m
\mapsto mc$ is a surjective homomorphism from $M'$ onto $cM \cap
Mc$. Hence, $cM \cap Mc$ with the $\circ$ multiplication is a divisor
of $M$, and so is $cMc \cup \smallset{c}$. The \emph{local divisor
  $M_c$ of $M$ at $c$} is the monoid $(cMc \cup \smallset{c}, {\circ},
c)$. Note that if $c = c^2$ is idempotent, then $\circ$ and the usual
multiplication in $M$ coincide for $cMc = cM \cap Mc$. So, we are in
accordance with the standard notation used for local divisors.  The
proofs in this paper would work equally well for $cM \cap Mc$ as for
$cMc \cup \smallset{c}$. In the definition of the local divisor $M_c$
we have given the preference to $cMc \cup \smallset{c}$, because it
might have fewer elements than $cM \cap Mc$.

The notion of local divisor can be generalized to transformation
monoids.  If $(X,M)$ is a transformation monoid and $c \in M$, then
there is a natural action $\circ$ of the local divisor $M_c$ on $Xc =
X\cdot c$ by $x c \circ cm = x \cdot cm = xc \cdot m$ for $xc \in Xc$
and $cm \in M_c$.

\begin{lemma}\label{lem:loctr}
  Let $(X,M)$ be faithful and $c \in M$. Then
  %$(Xc,
  %M_c)$ with the action $\circ$ is a transformation monoid strongly
  %dividing $(X,M)$. Moreover, if $(X,M)$ is faithful, then
  $(Xc,M_c)$
  is faithful, too.
\end{lemma}

\begin{proof}
  We have $x c \circ cm = x \cdot cm$ for all $x \in X$. Thus, the
  faithful action of $M$ on $X$ determines $cm\in M_c$.
\end{proof}

\begin{remark}\label{rem:local}
  The assertion of \reflem{lem:loctr} is one the main reasons why the
  tool of local divisors simplifies the proof of
  \refthm{thm:decomp}. In this proof we apply \reflem{lem:loctr} to a
  faithful and finite \tm $(X,M)$ where $c\in M$ is not a unit.  As a
  consequence, the action of $c$ is not a permutation. Thus, $\abs{Xc}
  < \abs{X}$. The submonoid $Mc \cup \smallset{1}$ of $M$ acts on
  $Xc$, and it makes sense to say that $(Xc,Mc \cup \smallset{1})$ is
  smaller than $(X,M)$, because $\abs{Xc} < \abs{X}$ and $\abs{Mc \cup
    \smallset{1}}Ê\leq \abs M$.  However, $(Xc,Mc \cup \smallset{1})$
  is never faithful!  Indeed, let $t, p \in \N$, $p>0$ such that
  $c^{t+p} = c^t$. Choose $t$ minimal with this property. Then $t\geq
  1$ since $c$ is not a unit.  We obtain $xc\cdot c^{t+p-1}= x\cdot
  c^{t} = xc\cdot c^{t-1}$.
  %% We also see that the semigroup $cMc\cup \smallset{c}\sse Mc $
  %% (which has always less elements than $M$ because $c$ is not a
  %% unit) cannot act faithfully on $Xc$ unless $t=1$.
  % For $t=1$ we have $c = c^{p+1}$. Then $cMc$ is a monoid with $c
  % \in cMc$ and neutral element $c^p$ (even if $p=1$ because then $c
  % = c^2$). In this case $(Xc,cMc)$ is indeed faithful: We have
  % $x\cdot c^p \cdot cmc = x\cdot cmc$ for all $x\in X$. Hence the
  % action determines the element $cmc$.
%
  \ERM
\end{remark}

The  interested reader may notice that the surjective mapping
\begin{equation*}
  \varphi \colon  X \to
  Xc, \, x \mapsto x\cdot c
\end{equation*}
yields $(Xc, M_c)\prec (X,M)$. More
precisely, the translation $x \mapsto x\cdot c$ defines a surjective
morphism of $(X,cM \cup \smallset{1})$ onto $(Xc,M_c)$.  This
justifies saying that $(Xc,M_c)$ is the \emph{local divisor of $(X,M)$
  at $c$}.

\section{A decomposition using local divisors}\label{sec:vdmain}

The following result is the main contribution of this paper.  Let $A$
be a generating set for the monoid $M$ and $c \in A$. It gives a
decomposition of $M$ into a wreath product of a local divisor $M_c$
and the submonoid $N = \gen{A \setminus \smallset{c}}$ generated by
$A\setminus \smallset{c}$, but this decomposition involves
constants. More precisely, if a \tm $(X,M)$ is faithful, then by
\reflem{lem:loctr} the transformation monoid $(Xc,M_c)$ is faithful,
too; and so we may assume that the monoids $M$, $M_c$, and $N$ are
submonoids of $T(X)$, $T(Xc)$, and $T(X \dotcup N)$,
respectively. Here $X \dotcup N$ denotes the disjoint union of $X$ and
$N$.

For a finite monoid $M$, by successively choosing $A$ to be minimal and
$c \in A$ not to be a unit, we will end up at groups with constants,
see \refcor{gwc}.

\begin{theorem}\label{thm:main}
  Let $(X,M)$ be a faithful transformation monoid such that $M$ is
  generated by $A$ and let $c \in A$.  Let $M_c$ be the local divisor
  of $M$ at $c$ and let $N = \gen{A \setminus \smallset{c}}$.  Then we
  have:
  \begin{equation*}
%    \overline{(X,M)} \ \text{ strongly divides } \
%    \overline{(X,M)} \ \text{ strongly divides } \
   \ov{(X,M)} \prec  \ov{(Xc,M_c)}
   \wr \ov{(X \dotcup N, N )}.
  \end{equation*}
\end{theorem}

\begin{proof}
  Let $M_c' = M_c \cup \overline{Xc}$, $X' = X \dotcup N$, and $N' = N
  \cup \overline{X\dotcup N}$.  We obtain $(Xc,M_c')= (Xc,M_c\cup \ov
  {Xc})$ and $(X',N')= (X \dotcup N, N \cup \ov{(X\dotcup N)})$.
  Let $W = M_c' \wr (X',N') =
  M_c'^{X'} \rtimes N'$. It  acts on  $Xc \times X'$ by
  $(p,y)\cdot(f,n') = (p\circ yf,y\cdot n')$.
  We define $\varphi\colon  Xc \times X' \to X$ by
  \begin{equation*}
    \varphi(p,y) = \begin{cases}
      y &\text{ if } y \in X \\
      p \cdot n & \text{ if } y = n  \in N.
    \end{cases}
  \end{equation*}
  We have to verify that this defines a division. For this, we have to
  show that every element in $A \cup \overline{X}$ has a cover.  A
  cover of $\overline{x} \in \overline{X}$ is $(f,\overline{x}) \in W$
  where $f$ is arbitrary. Then for all $(p,y) \in Xc \times X'$ we
  have
  \begin{equation*}
    \varphi\big( (p,y)\cdot (f,\overline{x}) \big)
    = \varphi( p \circ yf, x) = x = \varphi(p,y) \cdot \overline{x}.
  \end{equation*}
  In the following, we always assume $a \in A \setminus \smallset{c}$.
  A cover of $a$ is $(k_c,a)\in W$ and a cover of $c$ is
  $(f_c,\overline{1}) \in W$ where
  \begin{alignat*}{2}
    y k_c &= c && \quad \text{for all } y \in X' \\
    y f_c &= \overline{y \cdot c} && \quad \text{for } y \in X \\
    n f_c &= cnc && \quad \text{for } n \in N
  \end{alignat*}
  First, let $y \in X$. Then
  \begin{align*}
    \varphi\big( (p,y)\cdot (k_c,a) \big)
    &= \varphi(p\circ yk_c,y \cdot a) = y \cdot a = \varphi(p,y) \cdot a \\
    \varphi\big( (p,y)\cdot (f_c,\overline{1}) \big)
    &= \varphi(p\circ yf_c,1) = \varphi(p \circ \overline{y \cdot c},1)
    = \varphi(y \cdot c,1) = y \cdot c = \varphi(p,y) \cdot c.
%  \end{align*}
  \intertext{Let now $y = n \in N$. Then}
%  \begin{align*}
    \varphi\big( (p,n)\cdot (k_c,a) \big)
    &= \varphi(p\circ nk_c,na)
    = \varphi(p \circ c,na)
    = \varphi(p,na) = p \cdot na
    = \varphi(p,n) \cdot a \\
    \varphi\big( (p,n) \cdot (f_c,\overline{1}) \big)
    &= \varphi(p\circ nf_c,1) = \varphi(p \circ cnc,1)
    = \varphi(p\cdot nc,1) = p \cdot nc = \varphi(p,n) \cdot c.
  \end{align*}
  Therefore, every element in $A \cup \overline{X}$ has a cover in $W$
  which proves the claim.
\end{proof}

\begin{corollary}\label{gwc} %groups with constants
  Let $(X,M)$ be a transformation monoid such that $M$ is finite.
  Then we have
  \begin{equation*}
    (X,M) \prec \overline{(X_1,G_1)} \wr \cdots \wr \overline{(X_n,G_n)}.
  \end{equation*}
  Here $\abs {X_i} > 1$, every $(X_i,G_i)$ is faithful, and every $G_i$
  is a group dividing $M$ for all $1 \leq i \leq n$ and some $n\geq
  0$.  Moreover, the number $n$ can be chosen such that $\abs{G_1} +
  \cdots + \abs{G_n} < 2^{\abs{M}}$.
\end{corollary}

\begin{proof}
  Since $(X,M)$ strongly divides the faithful transformation monoid
  $(X \times M, M)$, we may assume that $(X,M)$ is faithful.
%vd (It is here where we use $X \neq \es$.)
  Since $(X,M)$ is faithful, we have $(X,M) \prec \ov{(X,M)}$.  Thus
  it suffices to prove that if $(X,M)$ is a faithful transformation
  monoid, then
  \begin{equation*}
    \overline{(X,M)} \prec \overline{(X_1,G_1)} \wr \cdots \wr \overline{(X_n,G_n)}.
  \end{equation*}
  where the $X_i$, $G_i$ and $n$ are as in the statement of the
  corollary.  We proceed by induction on $|M|$.  If $\abs X = 1$, then
  $M$ is trivial, too. We allow $n = 0$ in this case\footnote{By
    \refrem{rem:mon} the convention is that an empty wreath product
    defines the trivial \tm. Alternatively: for $\abs X = 1$ choose $n
    = 1$ and $X_1 = X\dotcup X$.}.  The assertion is trivial if
  $M$ is a group.
  Otherwise, let $A$ be a minimal generating set of $M$. Since $M$ is
  not a group, there exists a generator $c \in A$ which is not a
  unit. Let $N = \gen{A \setminus \smallset{c}}$.  We have
  $\abs{M_c}Ê< \abs M$ and $\abs{N}Ê< \abs M$; and we can apply
  \refthm{thm:main}. The result follows by induction.
\end{proof}

%
%\begin{corollary}\label{firstbound}
%Let $(X,M)$ be a a finite transformation monoid.
%Then the number $n$ in \refcor{gwc} can be chosen such that
%$\abs{G_1} + \cdots +
%  \abs{G_n} < 2^{\abs{M}}$. 
%\end{corollary}

%\begin{proof}As in the proof of \refcor{gwc} we may assume that $(X,M) = \ov{(X,M)}$. For $M= \smallset{1}$ the claim is true because $1<2^1$. 
%The result is now immediate by induction and \refthm{thm:main}. 
% \end{proof}

\begin{example}
  Let $[n]=\{1,\ldots,n\}$ and put $T_n=T([n])$.  A minimal generating
  set $A$ of $T_n$ consists of $a,b,c$ where $a$ is a transposition,
  $b$ is an $n$-cycle and $c$ is the idempotent sending $n$ to $n-1$
  and fixing all other elements.  Note that $[n]c=[n-1]$ and
  $cT_nc\cong T_{n-1}$.  Theorem~\ref{thm:main} then yields the
  decomposition $([n],T_n)\prec ([n-1],T_{n-1})\wr \overline{([n]\cup
    S_n,S_n)}$ where $S_n$ is the symmetric group on $[n]$.  Iteration
  yields the decomposition
  \[([n],T_n)\prec \overline{([2]\cup S_2,S_2)}\wr \overline{([3]\cup
    S_3,S_3)}\wr \cdots \wr \overline{([n]\cup S_n,S_n)}.\] This
  should be contrasted with the decomposition
  \[([n],T_n)\prec \overline{([2],S_2)}\wr \overline{([3],S_3)}\wr
  \cdots \wr \overline{([n],S_n)}\] given by the Holonomy
  Theorem~\cite{Eil76}.  In particular, our decomposition agrees with
  the Holonomy decomposition in number of factors but our construction
  blows up the state sets.

  On the other hand, the decomposition provided by the first proof of
  the Krohn-Rhodes theorem in Eilenberg~\cite{Eil76} is much worse.
  The first step gives a decomposition $([n],T_n)\prec
  ([n-1],T_ncT_n)\wr (S_n,S_n)$ and then decomposes $([n-1],T_ncT_n)$
  into a wreath product of several copies of $([n-1],T_{n-1})$, one
  for each of the $n$ left ideals generated by rank $n-1$ idempotents
  of $T_n$.  Our approach then seems to beat the previous inductive
  proofs.
  \EEX
\end{example}

\section{The Krohn-Rhodes decomposition}\label{sec:kr}

The Krohn-Rhodes theorem~\cite{kr65tams} was the first global
structure theorem in finite semigroup theory. Finite semigroups are
too general to be classified up to isomorphisms. One needs to use a
more global viewpoint to study them.  Groups embed in a wreath product
of their composition factors, which are certain simple group divisors.
One might hope that one could embed a finite semigroup into a wreath
product of composition factors of maximal subgroups and some
relatively small semigroups with only trivial maximal subgroups.  But
this is impossible since whenever $T(X)$ embeds into a semidirect
product, it embeds in one of the factors.  Thus one must introduce
division and obtain a decomposition only up to division, which is what
Krohn and Rhodes did.  This philosophy for ever changed finite
semigroup theory, leading to the current approach via varieties of
finite semigroups.  It also established the semidirect product as the
key player in the study of semigroup theory.  In summary the
Krohn-Rhodes theory is the closest thing to a Jordan-H\"older theorem
for semigroups.  It is also a powerful inductive scheme for proving
results about finite semigroups and regular languages, such as
Sch\"utzenberger's theorem on star-free languages \cite{sch65sf} (see
for example~\cite{cb68hawaii,Eil76,schutfromKR,rs09qtheory}).  For more
philosophy on the Krohn-Rhodes theorem, the reader is referred to the
book of Rhodes~\cite{WildBook}.

\begin{theorem}[Krohn/Rhodes~\cite{kr65tams}]\label{thm:decomp}
  Every finite transformation monoid $(X,M)$ strongly divides a wreath
  product of the form
  \begin{equation*}
    (X_1,M_1) \wr \cdots \wr (X_n,M_n)
  \end{equation*}
  where each factor $(X_i,M_i)$ is either $(\smallset{a,b},U_2)$ or it
  is of the form $(G,G)$ for some non-trivial simple group $G$ dividing $M$.
\end{theorem}

\begin{proof}
  By Corollary~\ref{gwc}, we can assume that $(X,M) =
  \overline{(X,G)}$ for some finite group $G$ and $\abs{X} > 1$. By
  \reflem{lem:dgwc}, this transformation monoid divides $(X,U_X) \wr
  (G,G)$.  By \reflem{lem:dUn}, $(X,U_X)$ divides a direct product of
  $\abs{X}-1$ copies of $(\smallset{a,b},U_2)$, which in turn divides
  a wreath product of $\abs{X}-1$ copies of $(\smallset{a,b},U_2)$ by
  \refex{exex}. By \refcor{sg}, $(G,G)$ divides a wreath product of
  simple groups $(G_1,G_1) \wr \cdots \wr (G_m,G_m)$ such that each
  $G_i$ divides $G$.
\end{proof}

Our approach to prove \refthm{thm:decomp} yields a simple way to bound the
number of necessary wreath products by a singly exponential function:

\begin{corollary}\label{cor:count}
  Let $(X,M)$ be a a finite transformation monoid.
  Then the number $n$ in \refthm{thm:decomp} can be chosen such that
  \begin{equation*}
    n < \abs{M}(\abs{M}+\abs{X})2^{\abs{M}}.
  \end{equation*}
\end{corollary}

\begin{proof}
  Since $(X,M) \prec (X\times M, M)$ and $(X\times M, M)$ is faithful,
  it is enough to show the formula $n < (\abs{M}^2
  +\abs{X})2^{\abs{M}}$ for faithful \tm{}s, only.  Moreover, if
  $(X,M)$ is faithful we have $(X,M) \prec \ov{(X, M)}$.  The
  assertion of \refthm{thm:main} yields a binary tree where $\ov{(X,
    M)}$ is the root, its left subtree is recursively defined by its
  root $\ov{(Xc,M_c)}$ and its right subtree by its root $\ov{(X
    \dotcup N, N )}$. Leaves are of the form $\ov{(X_i,G_i)}$ where
  $G_i$ is a group. The distance of such a leaf to the root is bounded
  by $\abs{M} - \abs{G_i}$. It also follows that we have:
  \begin{equation*}
    \abs{X_i} \leq \abs{X} + (\abs{M}-1) + (\abs{M}-2) + \cdots + \abs{G_i}
    \leq \abs{X} + \abs{M}(\abs{M}-1)/2.
  \end{equation*}
  
  Now, we continue by making each of the $\ov{(X_i,G_i)}$ to be inner
  nodes. Its left child is $(X_i,U_{X_i})$ its right child becomes
  $(G_i,G_i)$. This is the splitting according to \reflem{lem:dgwc}.
  We continue on the group side $(G_i,G_i)$ until all leaves are of
  the form $(G_i,G_i)$ where $G_i$ is simple. Due to \refcor{sg} the
  distance of all nodes to the root in this tree is still at most
  $\abs{M} - 1$. Since it is a binary tree we obtain at most
  $2^{\abs{M} - 1}$ leaves. In the worst case all leaves are now of
  the form $(X_i,U_{X_i})$ for which we need additional wreath
  products. However as we have $\abs{X_i} < \abs{X} + \abs{M}^2$,
  we conclude with \reflem{lem:dUn}.
 \end{proof}

\begin{remark}\label{rem:kr}
  For a moment let $\mathcal T$ be the ordered monoid of all (of
  isomorphism classes) of finite transformation monoids with the
  wreath product $\wr$ as multiplication and with strong division
  $\prec$ as ordering. Let $\mathcal P$ be the submonoid generated by
  the \tm{}s $(\smallset{a,b},U_2)$ and $(G,G)$ where $G$ is a
  non-trivial simple group.  \refthm{thm:decomp} can be rephrased by
  saying that for all $(X,M)\in \mathcal T$ there is some $(Y,K)\in
  \mathcal P$ such that $(X,M) \prec (Y,K)$. \refcor{cor:count} says
  that we need less than $\abs{M}(\abs{M}+\abs{X})2^{\abs{M}}$
  generators of $\mathcal P$ to express $(Y,K)$.
  \ERM
\end{remark}

\begin{remark}
  The Holonomy Theorem of~\cite{Eil76} provides a bound on the length
  of the decomposition in Corollary~\ref{cor:count} that is
  exponential in $|X|$ rather than $|M|$.  It therefore is a tighter
  result since in practice $|X|$ will be no bigger than $|M|$ as
  $(X,M)$ will come from an automaton in which all states are
  accessible from the initial state.  The improved bound is at the
  price of a more complicated proof.
  \ERM
\end{remark}

\begin{remark}
  It was proved by Krohn and Rhodes that the prime monoids are exactly
  the finite simple groups and the submonoids of $U_2$, where a monoid
  $M$ is prime if whenever it divides a semidirect product of two
  monoids, it divides one of the factors.  The situation for
  transformation monoids is more delicate and can be found in
  Eilenberg~\cite{Eil76}.  There the more general definition of
  division is used and it is not clear whether the prime
  transformation monoids are the same with respect to strong division.
  \ERM
\end{remark}

\section{Appendix: Missing proofs}\label{app}

%
%\begin{lemma}\label{lem:wrdiv}
%  If $(X,M)$ strongly divides $(X',M')$ and $(Y,N)$ strongly
%  divides $(Y',N')$, then the wreath product $(X,M) \wr (Y,N)$ strongly divides
%  $(X',M') \wr (Y',N')$.
%\end{lemma}
For convenience of the reader we repeat the statements where the
proofs have previously been missing.  The proof techniques are
well-known and not meant to be original. It should however be noted
that our divisor relation is based on totally defined surjective
mappings rather than on partially defined functions. Thus, we pay
attention to that.

\medskip

{\bf \noindent \reflem{lem:wrdiv}:} If $(X,M) \prec (X',M')$ and
$(Y,N) \prec (Y',N')$, then $(X,M) \wr
(Y,N) \prec (X',M') \wr (Y',N')$.

\begin{proof}
  Let the divisions be defined by the surjective functions $\varphi
  \colon X' \to X$ and $\varphi' \colon Y' \to Y$ and the surjective
  homomorphisms $\psi \colon \widehat{M} \to M$ and $\psi' \colon
  \widehat{N} \to N$ for submonoids $\widehat{M} $ of $ M'$ and
  $\widehat{N} $ of $ N'$, respectively.
  This induces a surjective function $\varphi \times \varphi' \colon
  X' \times Y' \to X \times Y$. Let $P$ contain all functions $f \in
  \widehat{M}^{\,Y'}$ satisfying $yf = y'f$ whenever $\varphi'(y) =
  \varphi'(y')$. The set $P$ is a submonoid of $\widehat{M}^{\,Y'}$
  and hence, it is a submonoid of $M'^{Y'}$.  Suppose $f \in P$ and $n
  \in \widehat{N}$. If $\varphi(y') = \varphi(y)$ for $y,y' \in Y'$,
  then
  \begin{equation*}
    \varphi'(y\cdot n) = \varphi'(y)\cdot \psi'(n)
    = \varphi'(y')\cdot \psi'(n) = \varphi'(y'\cdot n).
  \end{equation*}
  Since $y(n * f) = (y \cdot n)f = (y' \cdot n)f = y'(n * f)$
  by $f \in P$, it follows $n * f \in P$. Hence $\widehat{N} *
  P \subseteq P$ and $P \rtimes \widehat{N}$ is a submonoid of
  $M'^{Y'} \rtimes N'$.  We obtain a surjective homomorphism
  $\widetilde{\psi} \colon P \to M^Y$ with
  $\varphi(y)\widetilde{\psi}(f) = \psi(yf)$. By construction of $P$,
  this definition is independent of the choice of $y \in Y'$. For all
  $(x,y) \in X' \times Y'$ and for all $(f,n) \in P \times
  \widehat{N}$ we have
  \begin{align*}
    (\varphi \times \varphi') \big( (x,y) \cdot (f,n) \big)
    &= (\varphi \times \varphi')
    \big(x \cdot y f,y \cdot n \big) \\
    &= \big(\varphi(x \cdot yf),
    \varphi'(y \cdot n) \big) \\
    &= \big( \varphi(x) \cdot \psi(yf), \varphi'(y) \cdot \psi'(n) \big) \\
    &= \big( \varphi(x) \cdot \varphi'(y)\widetilde{\psi}(f),
    \varphi'(y) \cdot \psi'(n) \big)  \\
    &= (\varphi \times \varphi')(x,y) \cdot (\widetilde{\psi}(f),\psi'(n)).
%
%    \qedhere
  \end{align*}
  Thus $\varphi \times \varphi'$ and $\widetilde{\psi} \times \psi' \colon
 P \times \widehat{N} \to M^Y \rtimes N$ define a strong division.
\end{proof}

%\begin{proposition}\label{prp:dg}
{\bf \noindent \refprp{prp:dg}:}  If $N$ is a normal subgroup of $G$, then $(G,G) \prec (N,N) \wr (G/N,G/N)$.
%\end{proposition}

\begin{proof}
  Let $h_1,\ldots,h_n \in G$ be representatives of the cosets of $N$.
  By identifying cosets with their representatives, we can assume that
  $G/N$ acts on $\oneset{h_1,\ldots,h_n}$.  For each $g \in G$ let
  $[g] = h_i$ such that $Ng = N h_i$.  We define $\varphi \colon N
  \times \oneset{h_1,\ldots,h_n} \to G$ by $\varphi(n,h_i) = n h_i$.
  A cover of $g \in G$ is $(f_g,[g])$ where $hf_g = hg[hg]^{-1}$ for
  $h \in \oneset{h_1,\ldots,h_n}$.  Note that $hg[hg]^{-1} \in N$
  since $Nhg = N[hg]$. Now,
  \begin{equation*}
    \varphi\big( (n,h) \cdot (f_g,[g]) \big)
    = \varphi\big( n \cdot hf_g, \big[h [g]\big] \big)
    = \varphi( n hg [hg]^{-1}, [hg] )
    = n hg = \varphi(n,h) \cdot g.
%
%    \qedhere
  \end{equation*}
  Thus $(G,G)$ strongly divides $(N,N) \wr (G/N,G/N)$.
\end{proof}

%\begin{lemma}\label{lem:dgwc}
{\bf \noindent  \reflem{lem:dgwc}:}  
Let $(X,G)$ be a faithful transformation monoid such that $G$ is a
group. Then 
\begin{equation*}
  \overline{(X,G)} \prec (X,U_X) \wr (G,G).
\end{equation*}
% \end{lemma}

\begin{proof}
  Let $\varphi(x,g) = x \cdot g$ for all $(x,g) \in X \times G$. A
  cover of $g \in G$ is $(k_1,g)$ with $h k_1 = 1$ for all $h \in G$
  and a cover of $\overline{x} \in \overline{X}$ is $(f_x,1)$ with $h
  f_x = \overline{x \cdot h^{-1}}$. Now, for all $(y,h) \in X \times
  G$ we have
  \begin{align*}
    \varphi \big( (y,h) \cdot (k_1,g) \big)
    &= \varphi( y, hg ) = y \cdot hg = \varphi(y,h) \cdot g \\
    \varphi \big( (y,h) \cdot (f_x,1) \big)
    &= \varphi( y \cdot hf_x , h ) = \varphi( y \cdot
    \overline{x\cdot h^{-1}}, h) = \varphi( x\cdot h^{-1}, h )
    %= x h^{-1} h
    = x = \varphi(y,h) \cdot
    \overline{x}.
%
%    \qedhere
  \end{align*}
  Therefore, $\overline{(X,G)}$ strongly divides $(X,U_X) \wr (G,G)$.
\end{proof}

{\bf \noindent \reflem{lem:dUn}:}
%\begin{lemma}\label{lem:dUn}
  $(\smallset{a_0,\ldots,a_n},U_{n+1}) \prec
  (\smallset{a_0, \ldots, a_{n-1}},U_{n}) \times
  (\smallset{a_0,a_n},U_2)$.
%\end{lemma}

\begin{proof}
  Let $\varphi(a_k,a_\ell) = a_{\max(k,\ell)}$ for all $(a_k,a_\ell)
  \in \smallset{a_0, \ldots, a_{n-1}} \times \smallset{a_0,a_n}$. A
  cover of $\overline{a_i}$ with $i<n$ is
  $(\overline{a_i},\overline{a_0})$ and a cover of $\overline{a_n}$ is
  $(\overline{a_0},\overline{a_n})$.  Now, for all $(a_k,a_\ell) \in
  \smallset{a_0, \ldots, a_{n-1}} \times \smallset{a_0,a_n}$ we have
  \begin{align*}
    \varphi\big( (a_k,a_\ell) \cdot (\overline{a_i},\overline{a_0}) \big)
    &= \varphi( a_i, a_0 ) = a_i = \varphi( a_k, a_\ell ) \cdot \overline{a_i} \\
    \varphi\big( (a_k,a_\ell) \cdot (\overline{a_0},\overline{a_n}) \big)
    &= \varphi( a_0, a_n) = a_n = \varphi( a_k, a_\ell ) \cdot \overline{a_n}
%
%    \qedhere
  \end{align*}
  which proves the claim.
\end{proof}

%
%\bibliographystyle{fundam}
%\bibliography{../TRACES/traces}

{\small
\newcommand{\Ju}{Ju}\newcommand{\Ph}{Ph}\newcommand{\Th}{Th}\newcommand{\Ch}{C%
h}\newcommand{\Yu}{Yu}\newcommand{\Zh}{Zh}

}

\end{document}